\documentclass[a4paper,10pt]{article}

\usepackage{xcolor}
\usepackage{amsmath}
\usepackage{amssymb}
\usepackage{amsthm}
\usepackage[latin1]{inputenc}
\usepackage{graphicx}
\usepackage{subfigure}
\usepackage{eurosym}
\usepackage{lipsum}

\usepackage[pdfpagelabels=false,colorlinks=true,linkcolor=blue,citecolor=red,urlcolor=cyan]{hyperref}
\newtheorem{theorem}{Theorem}
\numberwithin{theorem}{section}
\newtheorem{proposition}[theorem]{Proposition}

\DeclareMathOperator*{\argmin}{arg\,min}

\newcommand{\Rr}{{\mathbb{R}}}

\newcommand{\Ii}{\mathcal I}

\newcommand{\Pp}{\mathcal P}

\title{Dual two-state mean-field games}

\author{Diogo Gomes$^1$, Roberto M. Velho$^2$, Marie-Therese Wolfram$^3$
\thanks{$^{1}$ CEMSE Division, 4700 King Abdullah University of Science and Technology, Thuwal 23955-6900, Kingdom of Saudi Arabia {\tt\small diogo.gomes@kaust.edu.sa}}%
\thanks{$^{2}$ CEMSE Division, 4700 King Abdullah University of Science and Technology, Thuwal 23955-6900, Kingdom of Saudi Arabia {\tt\small roberto.velho@gmail.com}}%
\thanks{$^{3}$  Radon Institute for Computational and Applied Mathematics, Austrian Academy of Sciences, Altenbergerstr. 69, 4040 Linz, Austria {\tt\small mt.wolfram@ricam.oeaw.ac.at}}}

\begin{document}
\thispagestyle{empty}
\pagestyle{empty}

\maketitle

\begin{abstract}
In this paper, we consider two-state mean-field games and its dual formulation. 
We then discuss numerical methods for these problems. 
Finally, we present various numerical experiments, exhibiting 
different behaviours, including shock formation, lack of invertibility, and monotonicity loss. 

\end{abstract}

\section{Introduction}

\noindent
The mean-field game framework is a flexible class of methods (see \cite{ll1,ll2, ll3, Caines1, Caines2})
with important applications in engineering, economics and social sciences. In this paper, we focus primarily
on two-state problems. 
These problems are among the simplest mean-field models which, nevertheless, have important applications.
These include, for instance, applications to socio-economic sciences, such as paradigm shift or
consumer choice behavior \cite{GVW-Socio-Economic}.
A variational perspective over two-state mean-field games was explored in \cite{gomes2011}. 
Finite state mean-field games can be formulated as systems of hyperbolic partial differential equations (see \cite{LCDF, GMS2}). A numerical method for these equations was introduced in \cite{GVW-Socio-Economic}. The key objective of this paper is to compare the outcome of different
numerical schemes for equivalent formulations of two-state mean-field games with a special emphasis in qualitative properties such as shock formation, 
invertibility and monotonicity loss.  

\noindent Consider a system of $N$ identical players or agents, which can switch between two distinct states.  
Each player is in a state $i\in\Ii = \{1, 2\}$ and can choose a switching strategy to the other state $j\in\Ii$.
The only information available to each player, in addition to its own state, is the fraction
$\theta_1$ and $\theta_2$ of players he/she sees in the different states $1$ and $2$.
 We define the probability vector 
$\theta = (\theta_1, \theta_2)$, $\theta\in \Pp(\Ii)=\{\theta\in \Rr^2\, :\, \theta_1 + \theta_2 =1,\, \theta_1,\theta_2 \geq 0\}$. As shown in 
 \cite{GMS2} the $N$ player mean-field game system admits a Nash equilibrium. 
 Furthermore, in the 
 limit $N\to\infty$, at least for short time,  
the value $U^i(\theta,t)$ for a player in state $i$, when the distribution of players among the different states
is given by $\theta$, satisfies the hyperbolic system
\begin{equation}
\label{hsys}
\left\{
\begin{aligned}
 - \ U_t^i(\theta,t) &= \sum_{j=1}^2 g_j(U, \theta) \ \frac{\partial U^i(\theta,t)}{\partial \theta_j}  + h(U,\theta,i),\\ 
				   U(\theta, T) &= U_T(\theta).
\end{aligned}
\right.
\end{equation}
Here $U^i:\Pp(\Ii)\times [0,T]\to \Rr$, 
$g:\Rr^2\times\Pp(\Ii)\to \Rr^2$,  
$h:\Rr^2\times\Pp(\Ii) \times \Ii \to \Rr$, $U_T^i:\Pp(\Ii)\to \Rr$ and $i \in \Ii=\{1,2\}$. 
The characteristics for \eqref{hsys} are a system of a finite state 
Hamilton-Jacobi equation, coupled with a transport equation for a probability measure, see \cite{GMS2}. 

\noindent Motivated by the detailed discussion in \cite{LCDF}, 
we consider the dual equation to \eqref{hsys},
which is
\begin{equation}\label{dhsys}
\left\{
\begin{aligned}
\Theta^i_t(\upsilon,t) &= g_i(\upsilon, \Theta) + \sum_{j=1}^2 h(\upsilon, \Theta, j) \ \frac{\partial \Theta^i(\upsilon, t)}{\partial \upsilon^j},\\
\Theta(\upsilon, T)&= \Theta_T(\upsilon),
\end{aligned}
\right.
\end{equation}
where $\Theta^i: \Rr^2 \times [0,T] \to \Rr$,  $\Theta_T^i: \Rr^2 \to \Rr$ and $i~\in~\Ii~=~ \{1,2\}$.
\noindent For certain classes of finite state mean-field games, called potential mean-field games, 
both \eqref{hsys} and \eqref{dhsys} can be regarded as gradients of a Hamilton-Jacobi equation \cite{LCDF, GMS2}. Thanks to a special reduction discussed in this paper,
this gradient structure can be used to study a wide range of two-state problems.

\section{Statement of the problem} \label{smfg}

\noindent
We begin by presenting the set up for two-state mean-field game problems. 
We assume that all players have the same running cost determined by a function
$c~:~\Ii~\times~\Pp(\Ii)~\times~(\Rr_0^+)^2~\to~\Rr$ as well as an identical terminal cost $U_T(\theta)$, which is
Lipschitz continuous in $\theta$.
The running cost $c(i, \theta, \alpha)$ depends on the state $i \in \Ii = \{1,2\}$ of the player, the mean-field $\theta$, that is the distribution of players among states, and on 
the switching rate $\alpha$. 
As in \cite{GMS2}, we suppose that $c$ is Lipschitz continuous in $\theta$, with a Lipschitz constant (with respect to $\theta$) bounded independently of $\alpha$.
Let the running cost $c$ be differentiable with respect to $\alpha$, and $\frac{\partial c}{\partial \alpha}(i, \theta, \alpha)$
be Lipschitz with respect to $\theta$, uniformly in $\alpha$. We assume that for each $i$, 
the running cost $c(i,\theta,\alpha)$ does not depend on the $i$-{th} coordinate $\alpha_{i}$ of $\alpha$.
Additional assumptions on $c$ are:
\begin{enumerate}
\item[(A1)] For $i=1,2$, $\theta\in \Pp(\Ii)$, $\alpha,  \alpha'\in (\Rr_0^+)^2$, with $\alpha_j \neq  {\alpha}_{j}'$, for
some $j\neq i$ and $\gamma > 0$,
\begin{equation*}
\hspace{-0.9cm}
c(i,\theta,\alpha\,')-c(i,\theta,\alpha)\geq 
\frac{\partial c(i,\theta,\alpha)}{\partial \alpha} \
  \cdot(\alpha\,'-\alpha)+\gamma\|\alpha\,'-\alpha\|^2. 
\end{equation*}
\item[(A2)] The function $c$ is superlinear on $\alpha_{j}$, $j \neq i$, that is,
\[
\lim_{\alpha_j\to\infty} \frac{c(i, \theta,  \alpha)}{\|\alpha\|}\to \infty, \ \ j \in \{1,2\}.
\]
\end{enumerate}
\noindent
The generalized Legendre transform of $c$ is given by
\begin{equation} 
\label{hami}
h(z, \theta, i)=\min_{\mu \in {(\Rr_0^+)^2}} c(i, \theta, \mu) + \mu \cdot \Delta_i z,
\end{equation}
with $z=(z^1,z^2)$ and $\Delta_i (\varphi^1,\varphi^2) = \left(\varphi^1-\varphi^i, \varphi^2 -\varphi^i \right)$, with $\varphi=(\varphi^1,\varphi^2) \in \mathbb{R}^2$.
The point where the minimum is achieved in \eqref{hami} is denoted by 
$\alpha^*$:
\vspace{-0.19cm}
\begin{align}\label{alpha_expression}
\alpha^*_j(z,\theta, i)=\argmin_{\mu\in {(\Rr_0^+)^d}}c(i, \theta, \mu) +\mu\cdot \Delta_i z, \text{ for } j \neq i.
\end{align}
If $h$ is differentiable with respect to $z$  we have
\begin{equation}
\label{otheralpha_expression}
\alpha_j^*(\Delta_iz,\theta,i) = \frac{\partial h\left(\Delta_iz,\theta,i\right)}{\partial z^j}, \text{ for } j \neq i.
\end{equation}
For convenience and consistency with \eqref{otheralpha_expression},
we require
\begin{equation} \label{eq:symmetry_alphas}
\alpha^*_1(z,\theta, i)= - \alpha^*_2(z,\theta, i).
\end{equation}

\noindent When the total number of players $N$ goes to infinity, we have the description of the Nash equilibrium given by the value function $U$ satisfying the hyperbolic system \eqref{hsys} where the function $g$ is given by
\begin{equation*} 
g_j(z,\theta)= \theta_1 \ \alpha^*_j(z, \theta, 1) + \theta_2 \ \alpha^*_j(z, \theta, 2).
\end{equation*}
Note that this yields $ g_1(z,\theta) = - \ g_2(z,\theta)$ by using \eqref{eq:symmetry_alphas}. Furthermore,
from \eqref{hami} and \eqref{alpha_expression}, we have that 
both $h(z,\theta,i)$ and $g(z,\theta)$ depend only on the difference $z^1-z^2$.

\subsection{Reduced primal problem}\label{reduction_scalar}
\noindent We can explore the particular structure of \eqref{hsys} for two-state problems
to transform it into a scalar problem. Observe that $\theta$ is a probability vector, so we rewrite $\theta$ as $\theta=(\theta_1,\theta_2)=(\zeta, 1-\zeta)$, $\zeta \in [0,1]$.
Let $U$ be a $C^1$ solution to \eqref{hsys}.
Because  $h(z,\theta,i)$ depends only on the differences of the coordinates of $z$, we
define $w(\zeta, t)=U^1(\zeta,1-\zeta,t)~-~U^2(\zeta,1-\zeta,t)$, 
and set $w_T(\zeta) =U_T^1(\zeta, 1-\zeta)-U_T^2(\zeta, 1-\zeta)$.
Thus, the hyperbolic system \eqref{hsys} is reduced to a scalar equation called the reduced primal equation, 
(see \cite{GVW-Socio-Economic}):
\begin{equation}
\label{scalar}
- w_t(\zeta,t) + r(w,\zeta) \ \partial_{\zeta} w(\zeta,t) = q(w,\zeta),
\end{equation}
where
\vspace{-0.28cm}
\begin{align*}
r(w,\zeta) & = - g_1(w,0,\zeta,1-\zeta),\\
q(w,\zeta) & = h(w,0,\zeta,1-\zeta,1) - h(w,0,\zeta,1-\zeta,2),
\end{align*}
and $\frac{\partial w}{\partial \zeta}$ denotes $ \frac{\partial w}{\partial \zeta} = \left( \frac{\partial}{\partial \theta_1} - \frac{\partial}{\partial \theta_2} \right) \left( U^1 - U^2 \right)|_{(\zeta,1-\zeta)}  $.
Note that \eqref{scalar} is supplemented with the terminal condition $w(\zeta, T)=w_T(\zeta)$, 
and since $r(w,0)\leq 0$ and $r(w, 1)\geq 0$ no further boundary conditions are required. 

\subsection{Dual problem}
\label{secdualhj}
\noindent Next we present a transformation introduced by Lions \cite{LCDF} to convert system \eqref{hsys} into an equivalent system of linear PDEs. 
This procedure is related to the hodograph transformation (see for instance \cite{E6}), an often used technique to convert certain nonlinear PDEs into 
a linear PDE by interchanging the dependent and independent variables. Note that this transformation is similar to the generalized coordinate techniques in classical Hamiltonian dynamics. 

\noindent For fixed time $t$, we consider the function $U(\theta, t)$, solution to~\eqref{hsys}, mapping from an open set of $\Rr^2$ into $\Rr^2$ and 
its inverse $\Theta(\upsilon, t) = (\Theta^1(\upsilon,t),\Theta^2(\upsilon,t))$, defined by
\[
\Theta(U(\theta, t), t)=\theta. 
\]
Using \eqref{hsys} we obtain that $\Theta(\upsilon,t)$ satisfies the dual system~\eqref{dhsys}.

\subsection{Reduced dual problem}
\noindent We now apply a similar reduction procedure to the dual problem. This reduction transforms the dual system \eqref{dhsys} into a scalar equation.
Let $\tilde{\upsilon} = \upsilon^1-\upsilon^2$. We consider the set where $\Theta^1+\Theta^2 = 1$ and deduce that
\begin{align*} 
\Theta^1_t &  = g_1(\tilde{\upsilon},0, \Theta^1, 1-\Theta^1) + h(\tilde{\upsilon},0, \Theta^1,1-\Theta^1, 1)  \ \frac{\partial \Theta^1}{\partial \upsilon^1} \\ \nonumber 
		   & \hspace{3.4cm} + h(\tilde{\upsilon},0, \Theta^1,1-\Theta^1,2)  \ \frac{\partial \Theta^1}{\partial \upsilon^2}.
\end{align*}
Next, we look for solutions depending only on $\tilde{\upsilon}$,
that is, $\Theta^1(\upsilon^1, \upsilon^2, t)=Z(\tilde{\upsilon}, t)$.  
The equation for $Z$ is given by
\begin{align*} 
Z_t = \left[h(\tilde{\upsilon},0, Z,1-Z, 1) - h(\tilde{\upsilon},0, Z,1-Z,2)\right] \  \frac{\partial Z}{\partial \tilde \upsilon} + g_1(\tilde{\upsilon},0, Z, 1-Z),
\end{align*}
which can be rewritten as
\begin{equation}\label{reduced_dualw}
 - Z_t(\tilde \upsilon,t) + q(\tilde \upsilon,Z) \ \frac{\partial Z}{\partial \tilde \upsilon} = r(\tilde \upsilon,Z).
\end{equation}
Equation \eqref{reduced_dualw} is supplemented with the boundary conditions: 
$\lim\limits_{\tilde \upsilon\to -\infty}Z(\tilde \upsilon, t)=1$ and 
$\lim\limits_{\tilde \upsilon\to +\infty}Z(\tilde \upsilon, t)=0$. These are motivated by the following considerations: 
if $\tilde{\upsilon}$ is very negative, the best state in terms of utility function is state $1$. Hence all players
would switch to it. Similarly, if $\tilde{\upsilon}$ is very large, then all players will switch to state $2$.


\subsection{Potential mean-field games}
\label{secpmfg}
\noindent We now consider a special class of mean-field games, called \emph{potential mean-field games},  in which system \eqref{hsys} can be written as
the gradient of a Hamilton-Jacobi equation. Suppose that
\begin{equation}
\label{seph}
h(z, \theta, i)=\tilde h(z, i)+f(i, \theta), \ \ \ i \in \{1,2\}.   
\end{equation}
Such functions $h$ that admit the decomposition as in \eqref{seph} will be called \emph{separable} throughout this paper. We will show in this section that separable mean-field games are potential. We are not aware of other classes of potential mean-field games. Separable mean-field games occur naturally in many problems. Various
examples in the realm of socio-economic sciences were discussed in \cite{GVW-Socio-Economic}. 
In this section, we suppose also that
$\displaystyle f(i, \theta)=\frac{\partial F(\theta) }{\partial \theta_i} $,
\vspace{+0.1cm}
 for some potential $F~:~\Rr^2~\to~ \Rr$. Define $ H:\Rr^2 \times \Pp(\Ii) \to \Rr $ by
\begin{equation}
\label{capH}
H(z, \theta)= \theta_1 \ \tilde h(\Delta_1 z, 1) + \theta_2 \ \tilde h(\Delta_2 z, 2) +F(\theta). 
\end{equation}
Let $\Psi_0: \mathbb{R}^2 \rightarrow \mathbb{R}$ be a continuous function and consider 
a smooth enough solution
$\Psi: \mathbb{R}^2 \times[0,T] \rightarrow \mathbb{R}$ of the Hamilton-Jacobi equation
\begin{equation}
\label{e:hj}
\begin{cases}
 \displaystyle -\frac{\partial \Psi(\theta,t)}{\partial t} = H \left(\partial_{\theta} \Psi, \theta \right), & \vspace{0.15cm} \\ 
\Psi(\theta, T) = \Psi_T(\theta).&
\end{cases}
\end{equation}
\noindent Setting $ \displaystyle U^j(\theta, t)=\frac{\partial \Psi(\theta, t)}{\partial \theta_j} $ we obtain that
\[
 -U^i_t = g_1 (U, \theta) \ \frac{\partial U^i}{\partial \theta_1} + g_2 (U, \theta) \ \frac{\partial U^i}{\partial \theta_2} + \tilde h(\Delta_i U, i) + \frac{\partial F(\theta)}{\partial \theta_i},
\]
and deduce 
that $U^i$ 
solves the PDE in \eqref{hsys}.

\subsection{Reduced potential mean-field games - I}
\noindent Note that the reduction to the scalar case performed in section \ref{reduction_scalar} can also be done in the potential case.
Once again, set $\theta=(\zeta, 1-\zeta)$, $\zeta \in [0,1]$, and define
\begin{equation*}
\Upsilon_T(\zeta) = \Psi_T(\zeta,1-\zeta).
\end{equation*}
Consider $\Upsilon: [0,1] \times [0,T] \to \Rr$ solution of the Hamilton-Jacobi equation 
\begin{equation} \label{eq:reduced_potential}
\begin{cases}
\displaystyle - \frac{\partial \Upsilon(\zeta,t)}{\partial t} = \tilde{H}(\partial_\zeta \Upsilon,\zeta),& \vspace{0.15cm}\\
\Upsilon(\zeta,T) = \Upsilon_T(\zeta), &
\end{cases}
\end{equation} 
where $ \tilde{H}: \Rr \times [0,1] \to \Rr $ has the form
\begin{equation}
\label{tildaga}
\tilde{H}(\partial_\zeta \Upsilon,\zeta)  = \zeta \ \tilde{h}(\partial_\zeta \Upsilon,0,1) + (1-\zeta) \ \tilde{h}(\partial_\zeta \Upsilon, 0,2) + F(\zeta,1-\zeta).
\end{equation}
Then $ \Psi (\zeta,1-\zeta,t) = \Upsilon(\zeta,t)$ solves \eqref{e:hj}.\\

\noindent The natural boundary conditions for \eqref{eq:reduced_potential}, taking into account that $\zeta\in [0,1]$,
are the state constrained boundary conditions, as discussed in \cite{GVW-Socio-Economic}. These can be implemented in practice
by taking large Dirichlet data for the boundary values of $\Upsilon$ at $\zeta=0,1$. 
 A solution to the reduced primal system can be addressed via the reduced potential system by setting 
$ w_p = \displaystyle \frac{\partial \Upsilon(\zeta,t)}{\partial \zeta} $ and observing that $w_p$ is a solution to \eqref{scalar}.

\subsection{Reduced potential mean-field games - II}

\begin{proposition}
A separable two-state mean-field game 
 has an associated reduced equation that admits a potential. 
\end{proposition}
\begin{proof}
Observing the expression of $q(w,\zeta)$ in the reduced primal formulation and using the fact that $h$ is separable we obtain:
\begin{equation*}
q(w,\zeta) = \left[ \tilde{h}(w,0,1) - \tilde{h}(w,0,2) \right] + \left[ f(1,\zeta) - f(2,\zeta) \right].
\end{equation*}
Since we are dealing with a problem in one dimension, $f(1,\zeta) - f(2,\zeta)$ is the derivative of some potential $\tilde{F}~:~\mathbb{R}~\to~\mathbb{R}$. So, without any additional assumptions, we conclude that the reduced equation admits a potential.
\end{proof}


\subsection{Potential formulation for dual systems}
\noindent Suppose \eqref{seph} holds and let $H$ be given by \eqref{capH}.
Fix $V_T~:~\Rr^2~\to~\Rr$ of class $C^1$ and take $V:\Rr^2\times [0, T] \to \Rr$ as a smooth solution to the dual Hamilton-Jacobi equation
\begin{equation}
\label{dualhj}
\begin{cases}
\displaystyle \frac{\partial V(\upsilon,t)}{\partial t}=H(\upsilon, \partial_\upsilon V),& \vspace{0.15cm} \\
V(\upsilon, T)=V_T(\upsilon).&
\end{cases} 
\end{equation}
Note that analogously to the primal case, the function $\Theta(\upsilon,t)=D_\upsilon V(\upsilon,t)$ solves the PDE in \eqref{dhsys}.

\subsection{Reduced Potential for dual systems}
As in the previous reduced cases, 
suppose
\[
 V_T(\upsilon^1,\upsilon^2)=\Phi_T(\upsilon^1-\upsilon^2). 
\]
Define $\Phi(\tilde{\upsilon},t)$ to be a solution to
\begin{equation}\label{eq:reduced_dual_potential}
\begin{cases}
\displaystyle \frac{\partial \Phi(\tilde{\upsilon},t)}{\partial t} = \tilde{H}(\tilde{\upsilon},\partial_{\tilde{\upsilon}} \Phi),& \vspace{0.15cm}\\
\Phi(\tilde{\upsilon},T) = \Phi_T(\tilde{\upsilon}), &
\end{cases}
\end{equation} 
where $ \tilde{H}: \Rr \times [0,1] \to \Rr $ has the form
\begin{align*}
\tilde{H}(\tilde{\upsilon},\partial_{\tilde{\upsilon}} \Phi)  = \partial_{\tilde{\upsilon}} \Phi \ \tilde{h}(\tilde{\upsilon},0,1) + (1-\partial_{\tilde{\upsilon}} \Phi) \ \tilde{h}(\tilde{\upsilon},0,2)  + F(\partial_{\tilde{\upsilon}} \Phi,1 - \partial_{\tilde{\upsilon}} \Phi).
\end{align*}
Then, it follows that $ V(\upsilon^1,\upsilon^2,t)=\Phi(\upsilon^1-\upsilon^2,t)$ solves the PDE in \eqref{dualhj}.

\noindent The boundary conditions associated to the dual problem suggest we should take boundary conditions for
$\Phi$ that are asymptotically linear. More precisely, 
\begin{equation*}
\lim_{\tilde{\upsilon} \to - \infty} \frac{\partial \Phi(\tilde{\upsilon},t)}{\partial \tilde{\upsilon}} = 1 \text{ and } \lim_{\tilde{\upsilon} \to + \infty} \frac{\partial \Phi(\tilde{\upsilon},t)}{\partial \tilde{\upsilon}} = 0.
\end{equation*}

\noindent Furthermore, a solution to the reduced dual equation can be constructed via the reduced potential dual equation \eqref{eq:reduced_dual_potential} by taking
 $Z_p = \frac{\partial \Phi(\tilde{\upsilon},t)}{\partial \tilde{\upsilon} }$, and observing that $Z_p$ solves the reduced dual equation \eqref{reduced_dualw}.

\subsection{Legendre transform}
\noindent Using the Legendre transform
as in \cite{LCDF},
one can relate the various terminal conditions for \eqref{hsys}, \eqref{dhsys}, \eqref{e:hj}, and \eqref{dualhj}. 

\noindent To do so, we fix a convex function $\Psi_T(\theta)$ and the corresponding solution $\Psi(\theta, t)$
of \eqref{e:hj}. Then
$\displaystyle U(\theta, t)=\frac{\partial \Psi(\theta,t)}{\partial \theta}$ solves \eqref{hsys} with terminal data
$\displaystyle  U_T(\theta)=\frac{\partial \Psi_T(\theta)}{\partial \theta}$.
To define the corresponding solutions to \eqref{dhsys} and \eqref{dualhj}
we consider the Legendre transform $V_T$ of $\Psi_T$:
\[
V_T(\upsilon)=\sup_{\theta} \upsilon\cdot \theta -\Psi_T(\theta). 
\]
So, by the usual properties of the Legendre transform, under sufficient regularity and convexity assumptions,
 the inverse of the map $\theta\mapsto \frac{\partial \Psi_T(\theta)}{\partial \theta}$ is
$\upsilon\mapsto \frac{\partial V_T(\upsilon)}{\partial \upsilon}$.\vspace{0.2cm}

\noindent Furthermore, 
as observed before, if we take the solution $V$ of \eqref{dualhj} with terminal data $V_T$,
its gradient $\Theta(\upsilon, t)~=~\frac{\partial V(\upsilon,t)}{\partial \upsilon}$ solves \eqref{dhsys}. 
Hence, the terminal data $\Theta_T$ is the inverse of $U_T$ and, at least for $t$ close enough to $T$, 
$\Theta(\upsilon, t)$ is also the inverse of $U(\theta, t)$, by the properties discussed previously. 
Besides, at least for $t$ close enough to $T$, $\Psi(\theta, t)$ is the Legendre transform of $V(\upsilon, t)$. 

\section{Examples and numeric simulations}\label{ses}
\noindent Consider a two-state mean-field game, where the fraction of players in either state, $1$ or $2$, is given by $\theta_i$, $i=1,2$ with $\theta_1+\theta_2 = 1$, and $\theta_i \geq 0$. 
Suppose the running cost $c = c(i,\theta,\mu)$ in \eqref{alpha_expression} depends quadratically on the switching rate $\mu$, i.e.,
\begin{align} \label{e:cost}
c(i,\theta,\mu) = f(i,\theta) + c_0(i,\mu), 
\end{align}
with $ c_0(i,\mu)~=~\frac{1}{2} \sum\limits_{j \neq i}^2 \mu_j^2 $.
Then $h$ and $g_1$ take the form
\begin{equation}\label{e:h}
\begin{aligned} 
h(z,\theta,1) &= f(1,\theta) - \frac{1}{2} \left((z^1-z^2)^+ \right)^2;\\
h(z,\theta,2) &= f(2,\theta) - \frac{1}{2} \left((z^2-z^1)^+ \right)^2;
\end{aligned}
\end{equation}
\vspace{-0.27cm}
\begin{equation}\label{formg1}
g_1(z,\theta) = - \ \theta_1  \left(z^1 - z^2\right)^+ + \theta_2 \ \left(z^2 - z^1\right)^+. 
\end{equation}
\noindent Note that if the function $f$ is a gradient field, i.e. $f=\nabla F$, the two-state problem is a potential mean-field game, cf. section \ref{secpmfg}. 
In this case, for $z=(z^1, z^2)\in \Rr^2$, \eqref{capH} is given by 
\begin{equation*} 
 H(z, \theta)=F(\theta_1,\theta_2)-\frac{\theta_1 \left((z^1-z^2)^+\right)^2+\theta_2 \left((z^2-z^1)^+\right)^2}{2}. 
\end{equation*}

\subsection{Some reduced systems}

Finally, we discuss the particular formulation of equation~\eqref{scalar} as well as the numerical simulations for different examples. Since $h$ is given by \eqref{e:h} and $g$ by \eqref{formg1}, we can rewrite the reduced equations for the primal and dual systems, as well for their potential versions. The reduced primal system becomes 
 \begin{align*} 
 - w_t(\zeta,t) - \frac{(1-2\zeta)|w| - w}{2} \ \frac{\partial w(\zeta,t)}{\partial \zeta}  = \frac{1}{2} \ |w| \ w   - \left[ f(1,\zeta,1-\zeta) - f(2,\zeta,1-\zeta) \right] 
 \end{align*}

\noindent The reduced dual system reads as
\begin{align*} 
- Z_t(\tilde{\upsilon},t) + \left( f(1,Z) - f(2,Z) -  \frac 1 2 |\tilde{\upsilon}| \tilde{\upsilon}   \right)  \frac{\partial Z(\tilde{\upsilon},t)}{\partial \tilde{\upsilon}} = \frac{(1-2Z) |\tilde{\upsilon}| - \tilde{\upsilon}}{2}.
\end{align*}
The $\tilde H$ given by \eqref{tildaga} is
\[
\tilde H(\tilde \upsilon, \zeta)=
-\frac{(\tilde{\upsilon}^-)^2+\zeta|\tilde{\upsilon}|\tilde{\upsilon}}{2}+F(\zeta, 1-\zeta).
\]

\noindent And their potential versions are given respectively by:
\begin{equation*}
\begin{cases}
\displaystyle - \frac{\partial \Upsilon(\zeta,t)}{\partial t} =
 \displaystyle - \frac 1 2 \left\lbrace \left[ { \left( \frac{\partial \Upsilon }{\partial \zeta} \right) }^-\right]^2
+\zeta \left|\frac{\partial \Upsilon }{\partial \zeta} \right|\frac{\partial \Upsilon }{\partial \zeta} 
 \right\rbrace + F(\zeta,1-\zeta), & \vspace{0.15cm} \\  
\Upsilon(\zeta,T) = \Upsilon_T(\zeta),& 
\end{cases}
\end{equation*}
the reduced potential formulation for the primal problem, and
\begin{equation*}
\begin{cases}
\displaystyle \frac{\partial \Phi(\tilde{\upsilon},t)}{\partial t} = - \frac 1 2 \left[ \left({\tilde{\upsilon}}^-\right)^2
+\frac{\partial \Phi}{\partial \tilde{\upsilon}} \left|\tilde{\upsilon}\right|\tilde{\upsilon}
\right]  + F\left(\frac{\partial \Phi}{\partial \tilde{\upsilon}},1- \frac{\partial \Phi}{\partial \tilde{\upsilon}}\right),& \vspace{0.15cm}\\
\Phi(\tilde{\upsilon},T) = \Phi_T(\tilde{\upsilon}), &
\end{cases}
\end{equation*}
the reduced potential formulation for the dual problem.

\subsection{Computational experiments}\label{sec:num_sim}
Finally we compare the numerical simulations of the primal, dual and potential mean-field game for different examples.
Let $\zeta \in \mathcal{I}$ denote the fraction of players being in state $1$. We discretize the domain $[0,1]$ into $N=200$ 
equidistant intervals. The time steps are set to $\Delta t=10^{-5}$, if not stated otherwise.\\
We solve the primal problem using the numerical discretization introduced in \cite{GVW-Socio-Economic}. The corresponding
potential mean-field game, the Hamilton-Jacobi equation \eqref{e:hj}, is solved using Godunov's method. The simulations of its dual formulation,
i.e. equation \eqref{reduced_dualw}, are based on a finite difference scheme using an upwind discretization for the convection term.

\paragraph{Example I - shock formation}

In our first example, we solve \eqref{scalar} by setting the terminal data to
\begin{equation*}
w(\zeta,T = 5) = 2 \zeta -1
\end{equation*}
and the running costs, as in \eqref{e:cost}, to
\begin{equation*}
f(1,\theta) = 1 - \theta_1 \text{ and } f(2,\theta) = 1 - \theta_2.
\end{equation*}
Hence $F(\theta) = \theta_1 \theta_2$.
We observe the formation of a shock in the primal version as well as in its corresponding potential version, see Figure \ref{f:ex1}. 
A boundary layer can also be seen in the dual variable $Z$. This results from
the discontinuities of $Z(\tilde{\upsilon}, T)$
at the boundary due to the limiting boundary conditions.
\begin{figure}[htb!]
\begin{center}
\subfigure[$w$ - Solution to the Primal problem at times $t=0$ and $t=5$.]{\includegraphics[width=0.75\columnwidth]{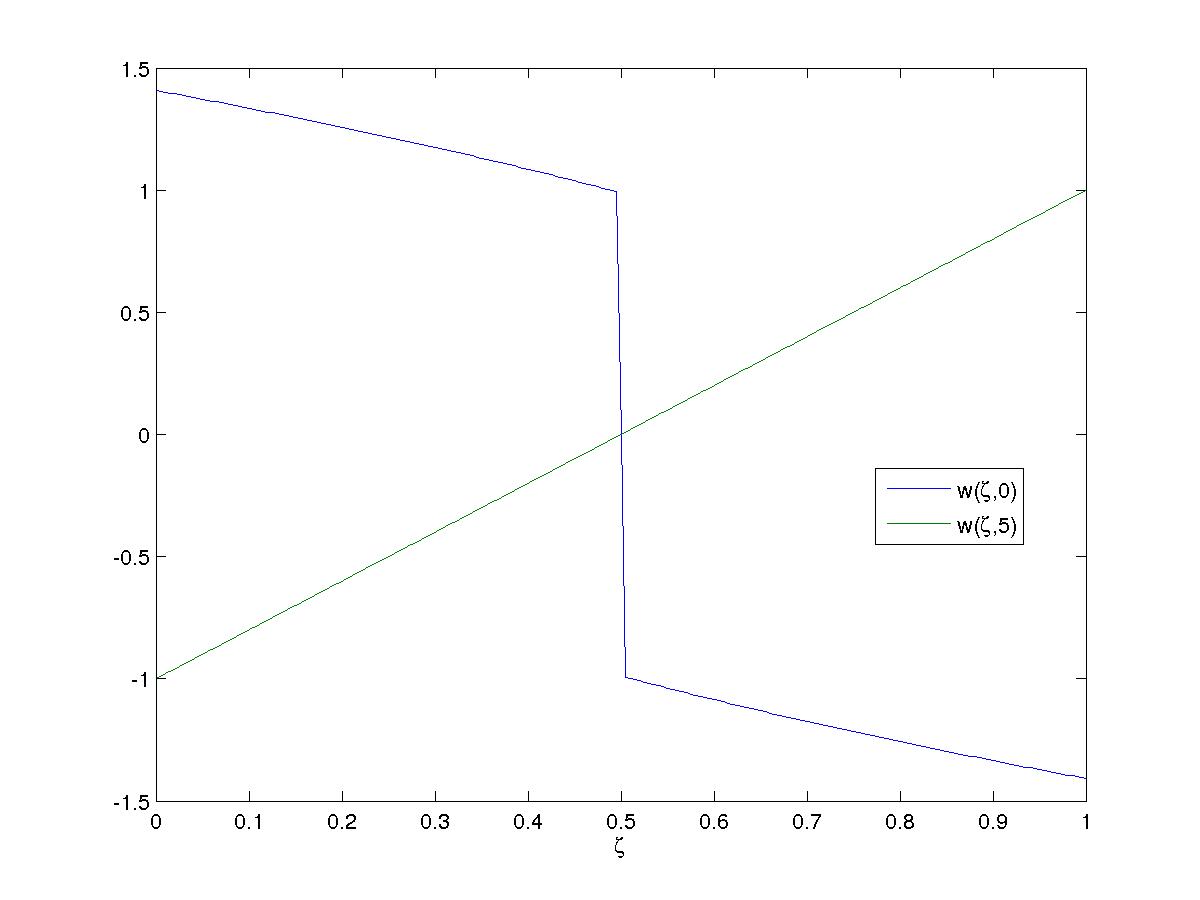}}
\subfigure[$Z$ - Solution to the Dual problem.]{\includegraphics[width=0.75\columnwidth]{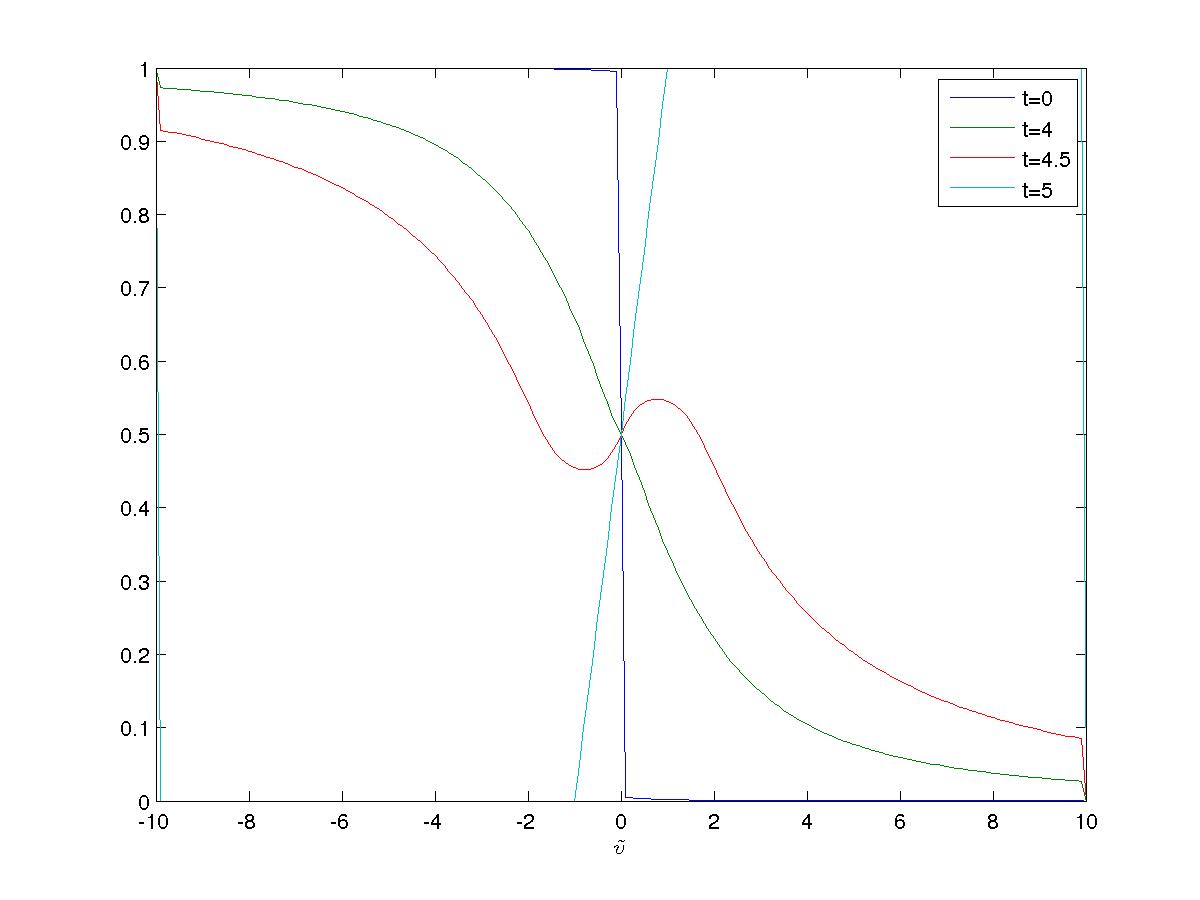}}
\subfigure[Zoom of $Z$ - Dual solution.]{\includegraphics[width=0.75\columnwidth]{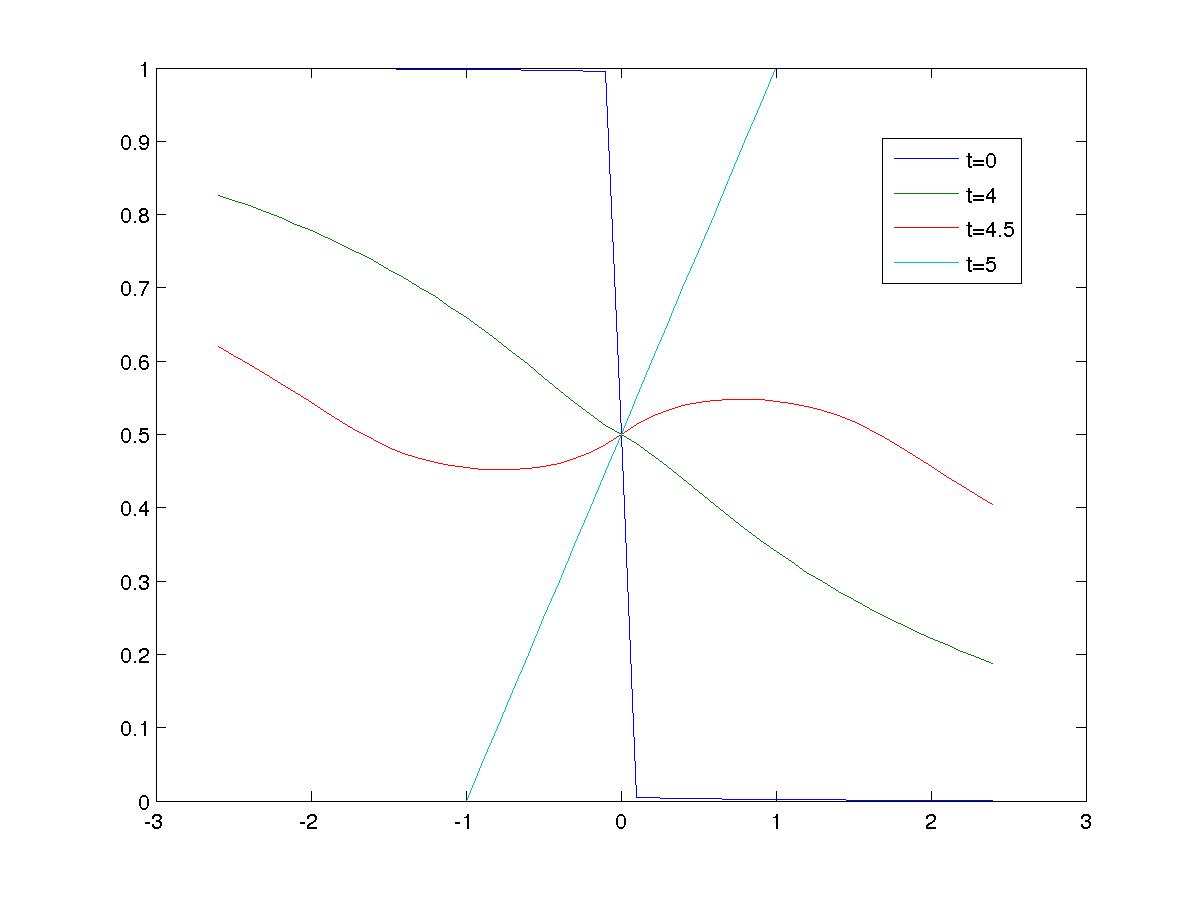}}
\caption{Simulations for Example~I.}\label{f:ex1}
\end{center}
\end{figure}

\begin{figure}[htb!]
\begin{center}
	\subfigure[Solution to the Reduced Primal ($w_p$) via the Reduced Potential Primal ($\Upsilon$) at time $t~=~0$.]{\includegraphics[width=0.5\columnwidth]{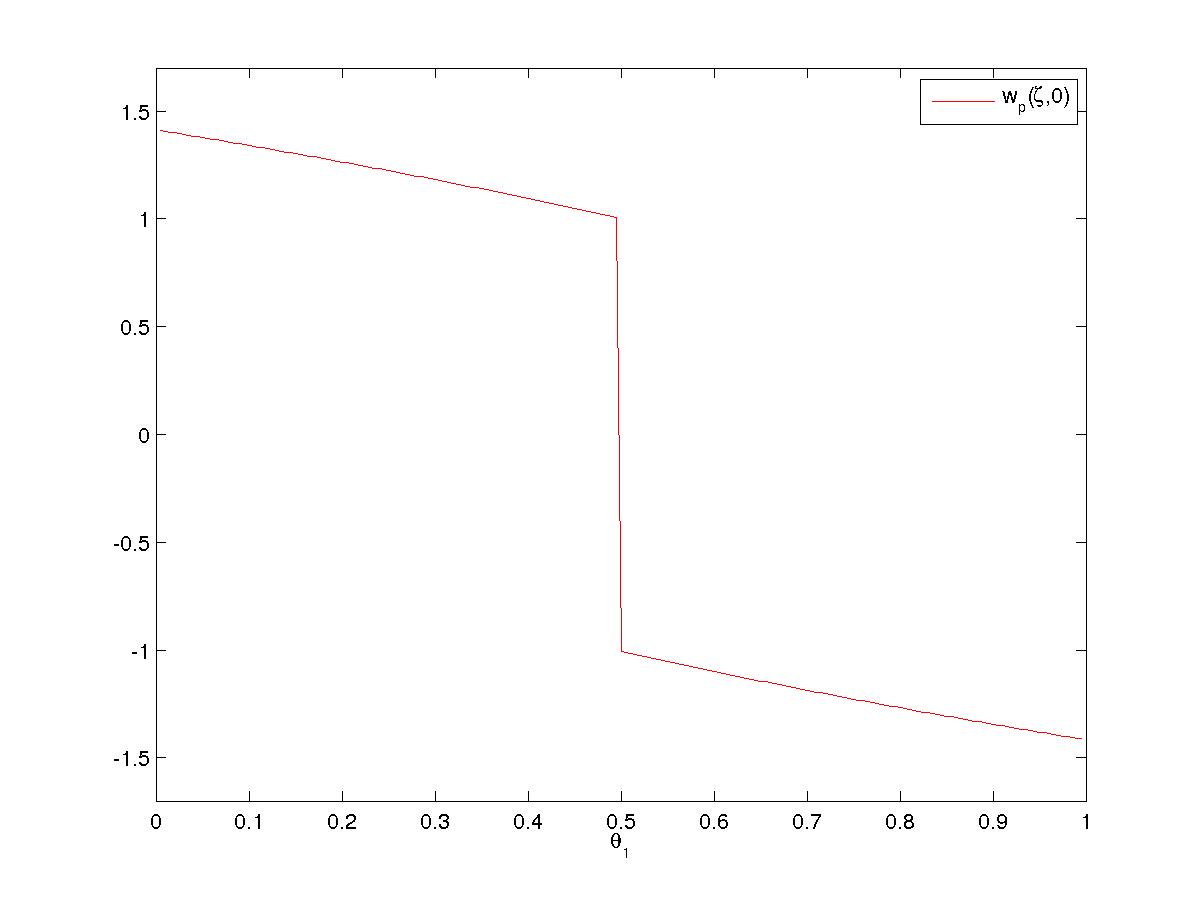}}
	\subfigure[Reduced Potential ($\Phi$) for the Dual Problem at time $t=0$.]{\includegraphics[width=0.5\columnwidth]{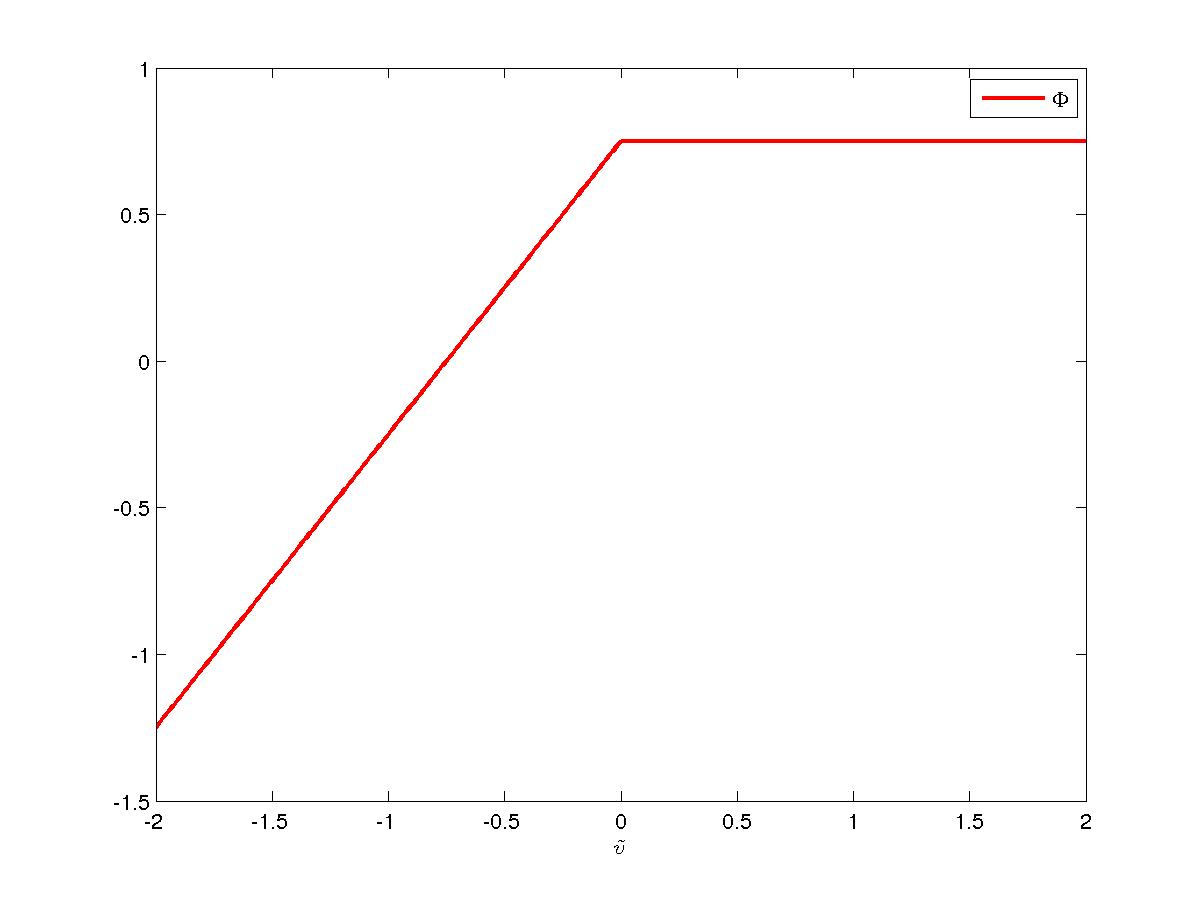}}
	\subfigure[Comparison of the solutions $w$, $w_p$ to the primal problem at time $t=0$.]{\includegraphics[width=0.5\columnwidth]{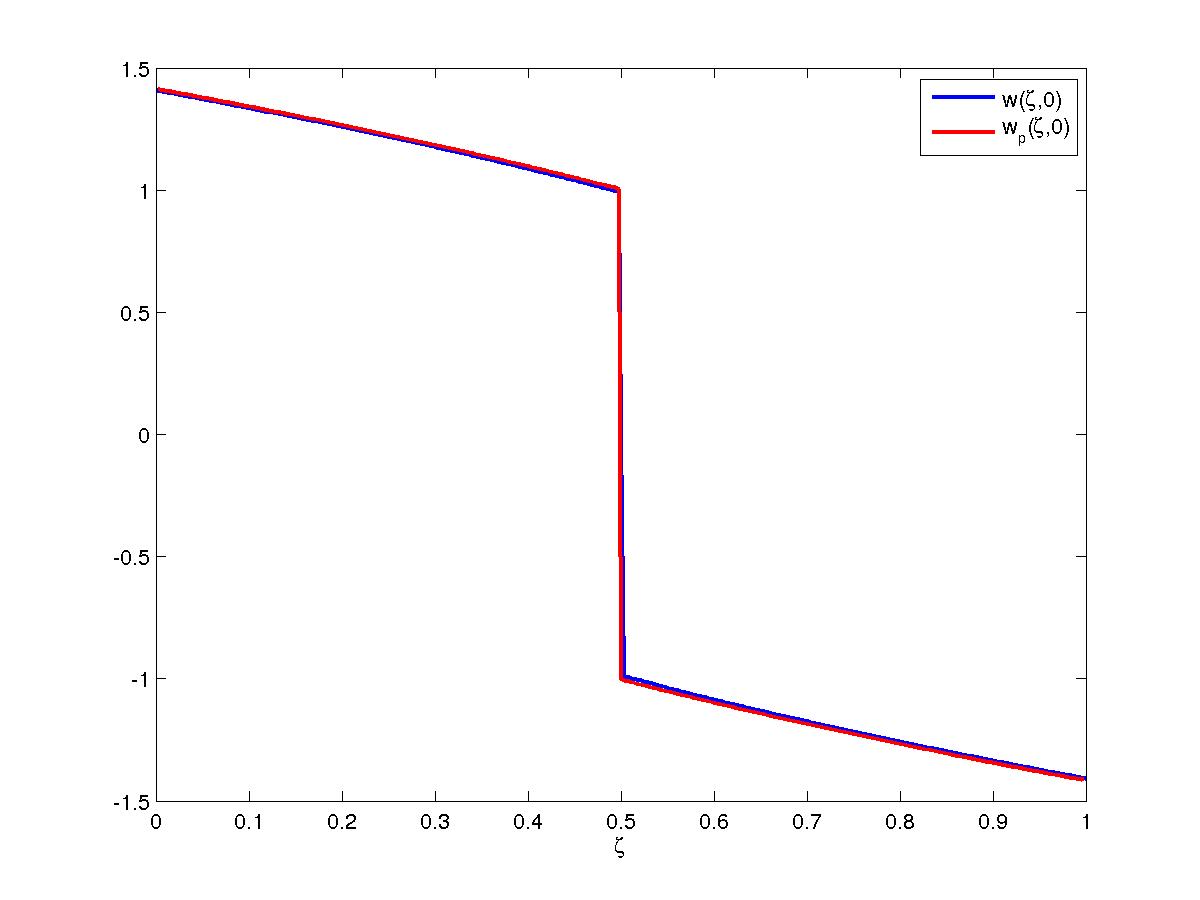}}
	\subfigure[Comparison of the solutions $z$ and $z_p$ to the dual problem at time $t=0$.]{\includegraphics[width=0.5\columnwidth]{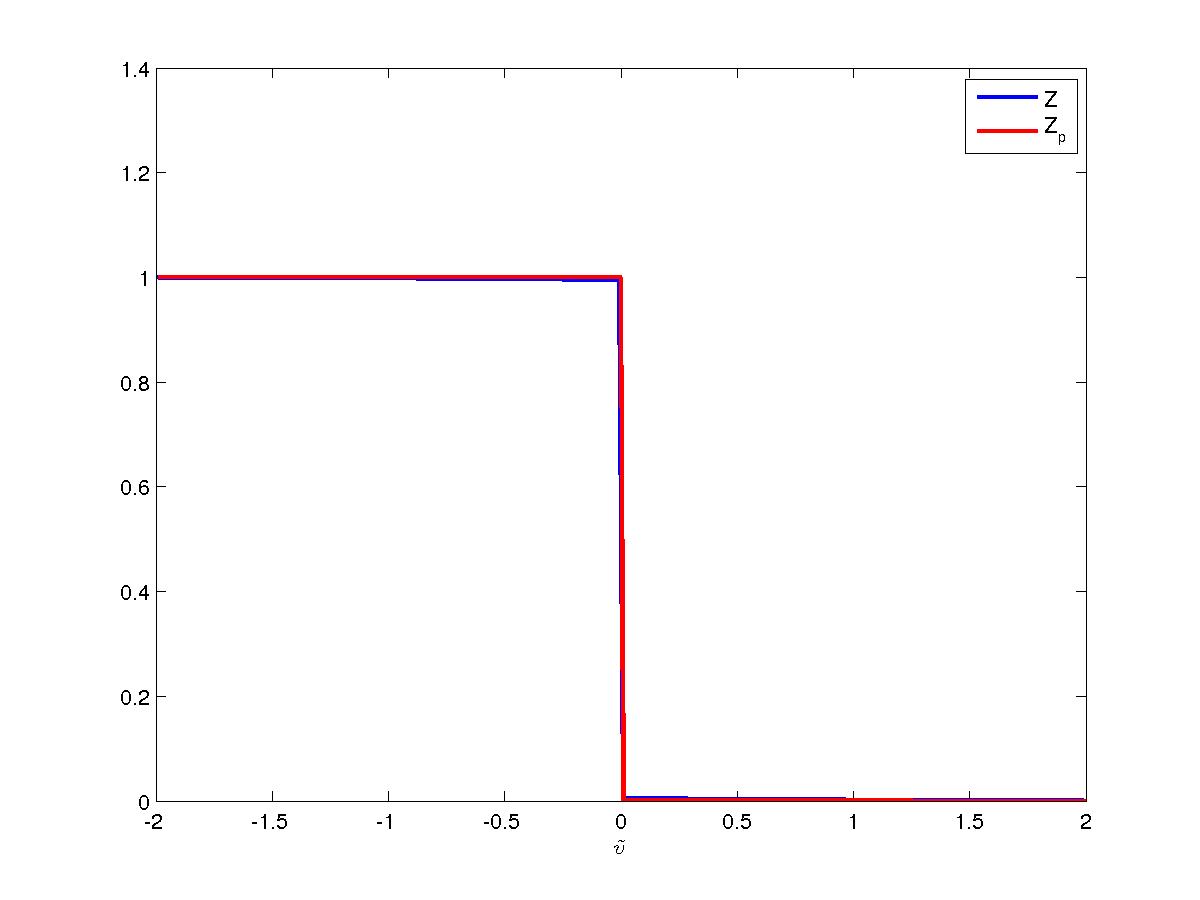}}
	\caption{Simulations for Example~I.} 
\end{center}
\end{figure}

\paragraph{Example II - monotonicity loss}
In our second example, we illustrate the behavior of solutions when $w$ loses its monotone behavior. In this case, the function is not invertible any more, hence we expect different shocks in the dual variable.\\
We choose 
\begin{align*}
F(\theta) = \kappa~\theta_1^2 \theta_2^2,~\kappa\in \mathbb{R}^+.
\end{align*}
Then $f(1,\theta) = 2\kappa~\theta_1^2\theta_2$ and $f(2,\theta) = 2\kappa~\theta_1 \theta_2^2$. The terminal
conditions are set to
\begin{align*}
w(\zeta,T = 0.25) = 2 \zeta - 1.
\end{align*}
We clearly observe the loss of monotonicity of $w$ at time $t=0$ in Figure \ref{f:ex3}. In this case, it is
not possible to invert the function $w$ any more. The formation of a discontinuity is also visible
in the evolution for $Z$. 
\begin{figure}[htb!]
\begin{center}
\subfigure[$w$ - Solution Primal problem.]{\includegraphics[width=0.75\columnwidth]{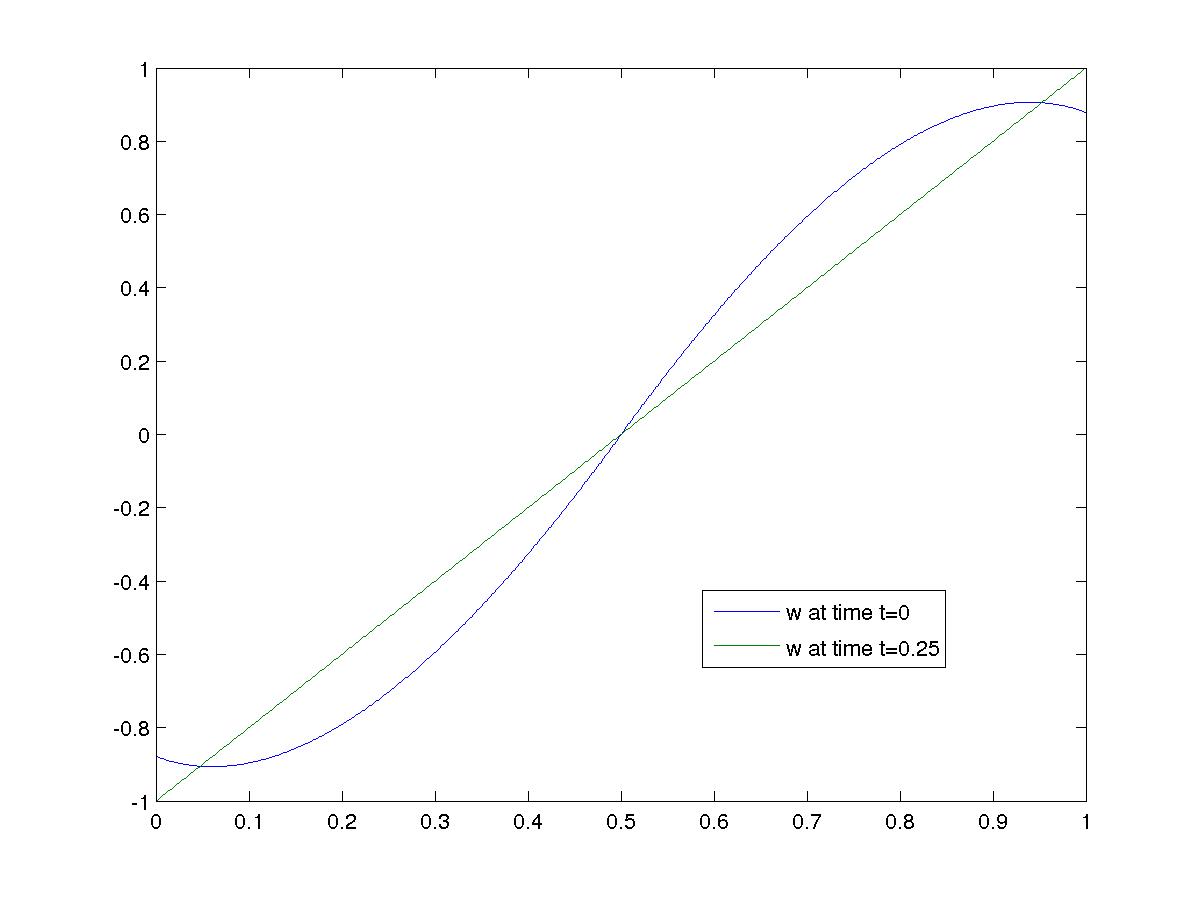}}
\subfigure[$Z$ - Solution Dual problem.]{\includegraphics[width=0.75\columnwidth]{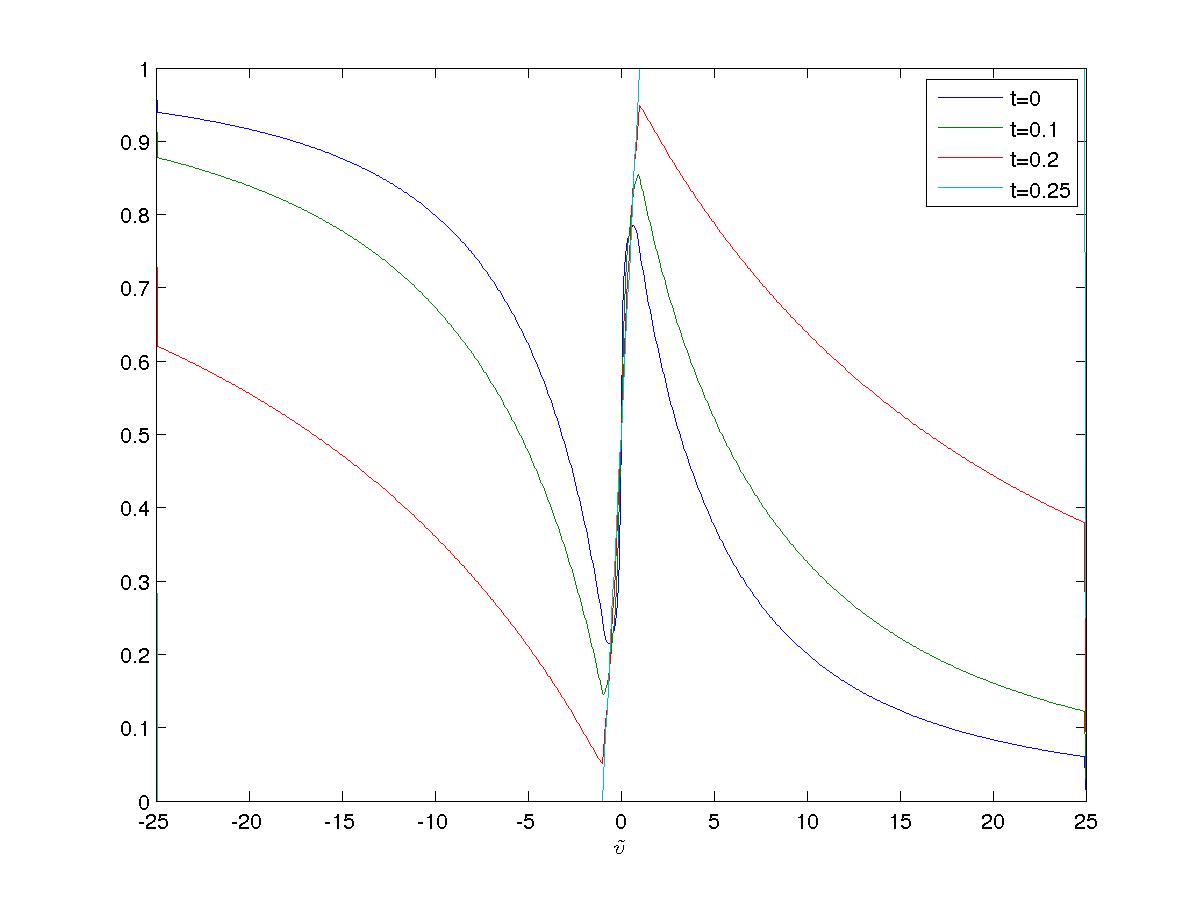}}
\subfigure[Zoom of $Z$ - Dual Solution.]{\includegraphics[width=0.75\textwidth]{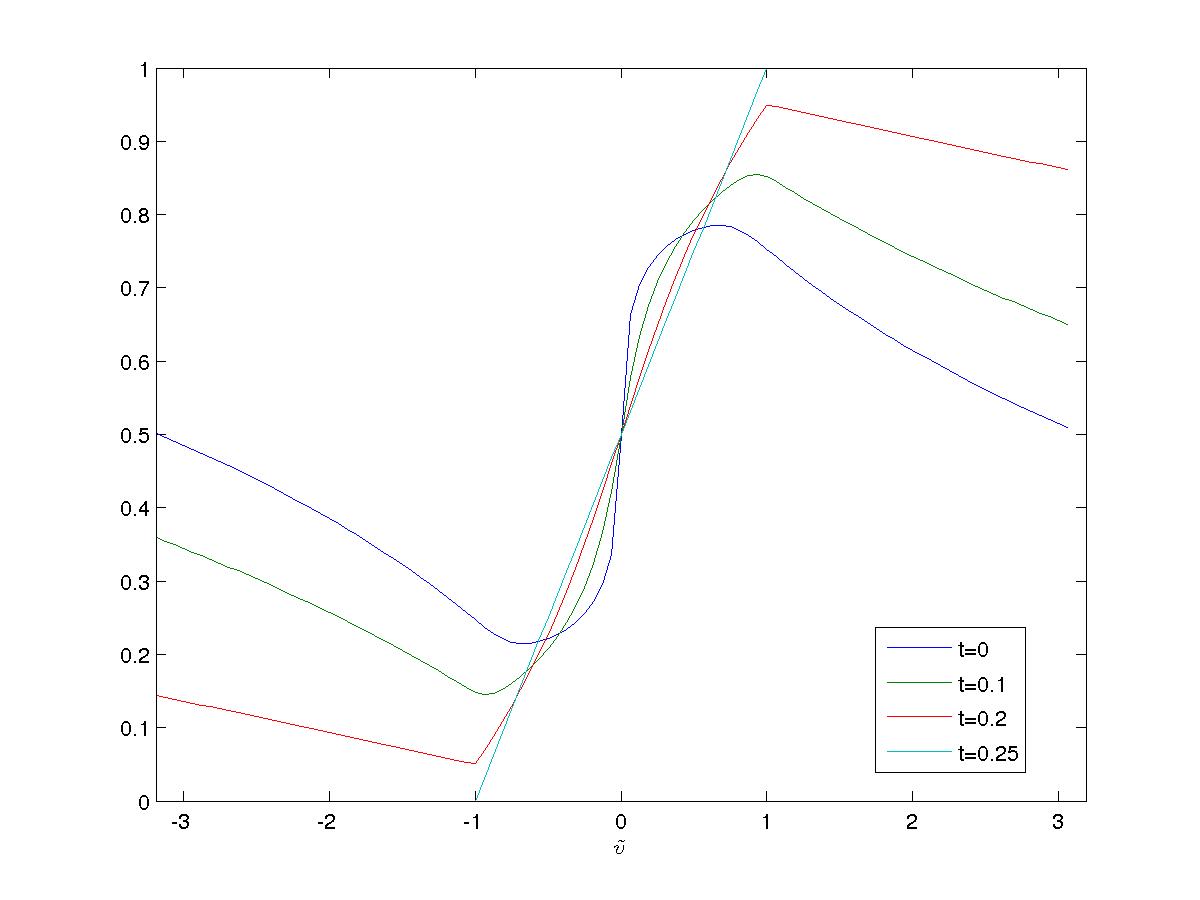}}
\caption{Simulations for Example~II.}\label{f:ex3}
\end{center}
\end{figure}

\section{Conclusions}
In this paper, we have examined the dual formulation for finite state mean-field games with particular emphasis on two-state problems where various reductions and simplifications are possible.  In particular,
we have shown that any separable two-state mean-field game admits a potential.
Additionally, the analysis of the boundary conditions for these problems was carried out in detail.
We have illustrated numerically the connection between shock formation, in one formulation, with the monotonicity loss in its dual formulation. For potential mean-field games, this corresponds to convexity/concavity loss of the associated potential functions.

\bibliographystyle{alpha}
\bibliography{mfg_urls}

\begin{thebibliography}{GVW14}

\bibitem[Eva98]{E6}
L.~C. Evans.
\newblock {\em Partial Differential Equations}.
\newblock Graduate Studies in Mathematics. American Mathematical Society, 1998.

\bibitem[GMS13]{GMS2}
D.~Gomes, J.~Mohr, and R.~R. Souza.
\newblock Continuous time finite state mean-field games.
\newblock {\em Appl. Math. and Opt.}, 68(1):99--143, 2013.

\bibitem[Gom11]{gomes2011}
D.A. Gomes.
\newblock Continuous time finite state space mean field games - a variational
  approach.
\newblock {\em 2011 49th Annual Allerton Conference on Communication, Control,
  and Computing, Allerton 2011}, pages 998--1001, 2011.

\bibitem[GVW14]{GVW-Socio-Economic}
D.~Gomes, R.~M. Velho, and M.-T. Wolfram.
\newblock Socio-economic applications of finite state mean field games.
\newblock {\em Preprint}, 2014.
\newblock
  \href{http://arxiv.org/abs/1403.4217}{http://arxiv.org/abs/1403.4217}.

\bibitem[HCM07]{Caines2}
M.~Huang, P.~E. Caines, and R.~P. Malham{\'e}.
\newblock Large-population cost-coupled {LQG} problems with nonuniform agents:
  individual-mass behavior and decentralized {$\epsilon$}-{N}ash equilibria.
\newblock {\em IEEE Trans. Automat. Control}, 52(9):1560--1571, 2007.
\newblock
  \href{http://dx.doi.org/10.1109/TAC.2007.904450}{http://dx.doi.org/10.1109/TAC.2007.904450}.

\bibitem[HMC06]{Caines1}
M.~Huang, R.~P. Malham{\'e}, and P.~E. Caines.
\newblock Large population stochastic dynamic games: closed-loop
  {M}c{K}ean-{V}lasov systems and the {N}ash certainty equivalence principle.
\newblock {\em Commun. Inf. Syst.}, 6(3):221--251, 2006.
\newblock
  \href{http://projecteuclid.org/getRecord?id=euclid.cis/1183728987}{http://projecteuclid.org/getRecord?id=euclid.cis/1183728987}.

\bibitem[Lio11]{LCDF}
P.-L. Lions.
\newblock College de france course on mean-field games.
\newblock 2007-2011.

\bibitem[LL06a]{ll1}
J.-M. Lasry and P.-L. Lions.
\newblock Jeux \`a champ moyen. {I}. {L}e cas stationnaire.
\newblock {\em C. R. Math. Acad. Sci. Paris}, 343(9):619--625, 2006.

\bibitem[LL06b]{ll2}
J.-M. Lasry and P.-L. Lions.
\newblock Jeux \`a champ moyen. {II}. {H}orizon fini et contr\^ole optimal.
\newblock {\em C. R. Math. Acad. Sci. Paris}, 343(10):679--684, 2006.

\bibitem[LL07]{ll3}
J.-M. Lasry and P.-L. Lions.
\newblock Mean field games.
\newblock {\em Jpn. J. Math.}, 2(1):229--260, 2007.

\end{thebibliography}

\section*{Acknowledgements}
DG was partly supported by
KAUST baseline and start-up funds, 
KAUST SRI, Uncertainty Quantification Center in Computational Science and Engineering, and CAMGSD-LARSys (FCT-Portugal).
RMV was partially supported by CNPq - Brazil through a PhD scholarship - Program Science without Borders and KAUST - Saudi Arabia. MTW acknowledges support from the Austrian Academy of Sciences \"OAW via the New Frontiers Project NST-0001.

\end{document}